\documentclass[12pt,reqno]{amsart}
\usepackage[english]{babel}
\usepackage{amssymb,amsmath,amscd,amsthm}
\usepackage{amsbsy}
\usepackage[latin1]{inputenc}
\usepackage{inputenc}
\usepackage{graphicx}
\newtheorem{thm}{Theorem}[section]
\newtheorem{cor}[thm]{Corollary}

\newtheorem{prop}[thm]{Proposition}
\newtheorem{defn}[thm]{Definition}

\newtheorem{rem}[thm]{Remark}
\numberwithin{equation}{section}
\DeclareMathOperator{\Hom}{Hom}
\DeclareMathOperator{\Der}{Der}
\DeclareMathOperator{\GL}{GL}
\DeclareMathOperator{\Orb}{Orb}
\DeclareMathOperator{\Leib}{\mathcal{L}\emph{eib}}
\DeclareMathOperator{\LN}{\mathcal{L}\emph{N}}
\DeclareMathOperator{\LR}{\mathcal{L}\emph{R}}
\DeclareMathOperator{\lev}{lev}

\begin{document}

\title[On the degenerations of solvable Leibniz algebras]{On the degenerations of solvable Leibniz algebras}

\author{J. M. Casas, A. Kh. Khudoyberdiyev, M. Ladra, B. A. Omirov}
\address{[J.\ M.\ Casas] Department of Applied Mathematics I, E. E. Forestal, University of Vigo, 36005 Pontevedra, Spain.} \email{jmcasas@uvigo.es}
\address{[M.\ Ladra] Department of Algebra, University of Santiago de Compostela, 15782, Spain.} \email{manuel.ladra@usc.es}

\address{[A.\ Kh.\ Khudoyberdiyev and B.\ A.\ Omirov] Institute of Mathematics, National University of Uzbekistan,
Tashkent, 100125, Uzbekistan.} \email{khabror@mail.ru, omirovb@mail.ru}

\begin{abstract}

The present paper is devoted to the description of rigid solvable
Leibniz algebras. In particular, we prove that solvable Leibniz
algebras under some conditions on the nilradical are rigid and we describe four-dimensional solvable Leibniz algebras with three-dimensional rigid nilradical.  We show that the Grunewald-O'Halloran's conjecture ``any $n$-dimensional nilpotent Lie algebra is a degeneration of some algebra of the same dimension'' holds for Lie algebras of dimensions less than six  and for Leibniz algebras of dimensions less than four. The algebra of level one, which is
omitted in the 1991 Gorbatsevich's paper, is indicated.

\end{abstract}


\thanks{}

\maketitle

\section{Introduction}

Largely because of their importance to string theory, quantum field theory and other branches of fundamental research in mathematical physics, noncommutative analogs of many classical constructions have received much attention in the past
few years \cite{EF, GS}.

The noncommutative analog of Lie algebras are Leibniz algebras, discovered by Loday when he handled periodicity phenomena in algebraic $K$-theory \cite{Lo}. This algebraic structure found  applications  in several fields as Physics and Geometry \cite{EFO, KW, Lod 1, Lod 2}.

Important subjects playing a relevant role in Mathematics and Physics are degenerations, contractions and deformations of Lie and Leibniz algebras. Namely, in  \cite{Seg} the notion of contractions of Lie algebras on physical grounds was  introduced:
if two physical theories (like relativistic and classical mechanics) are related by a limiting
process, then the associated invariance groups (like the Poincaré and Galilean groups)
should also be related by some limiting process. If the velocity of light is assumed to
go to infinity, relativistic mechanics ``transforms'' into classical mechanics. This also induces
a singular transition from the Poincaré algebra to the Galilean one.

Other example is a limiting process from quantum mechanics to classical mechanics under $\hbar \to 0$ of the Planck constant, that
corresponds to the contraction of the Heisenberg algebras to the abelian ones of the same
dimensions \cite{Druhl}.

Nevertheless, as it was proved in \cite{WW}, the notions of deformations, contractions and degenerations are isomorphic over
 the fields $\mathbb{R}$ or $\mathbb{C}$. Degenerations of Lie and Leibniz algebras were the subject of  numerous papers,
  see for instance \cite{Bal, Burde1, Burde2, Gorb,RA, Seeley, WW} and references given therein, and their research continues actively.
   These facts motivate that we focus our attention in the study of degenerations of solvable Leibniz algebras.

In order to do so, we know that an $n$-dimensional Leibniz algebra may be
considered as an element $\lambda$ of the affine variety
$\Hom(V\otimes V,V)$ via the  mapping $\lambda \colon V\otimes
V\to V$ defining the Leibniz bracket on a vector space $V$ of dimension $n$. Since
Leibniz algebras are defined via polynomial identities, the set of $n$-dimensional
Leibniz algebra structures,  $\Leib_n$, forms an algebraic subset of the
variety $\Hom(V\otimes V,V)$ and the linear reductive group
$\GL_n(F)$ acts on $\Leib_n$ via change of basis, i.e.,
\[(g*\lambda)(x,y)=g\Big(\lambda\big(g^{-1}(x),g^{-1}(y)\big)\Big), \quad  g \in \GL_n(F), \  \lambda \in \Leib_n .\]

The orbits $\Orb(-)$ under this action are the  isomorphism
classes of algebras. Note that solvable (respectively, nilpotent)
Leibniz algebras of the same dimension also  form an invariant
subvariety of the variety of Leibniz algebras under the mentioned
action.

Let $V$ be an $n$-dimensional vector space over a field $F$. The bilinear maps $V\times V \to V$ form an
$F^{n^3}$-dimensional affine space. We shall
consider the Zariski topology on this space. Recall, a set is
called irreducible if it cannot be represented as a union of two
nontrivial closed subsets, otherwise it is called reducible. The
maximal irreducible closed subset of a variety is called an
irreducible component. From algebraic geometry we know that an
algebraic variety is a union of irreducible components and that  closures
of open sets produce irreducible components. Therefore, for the
description of a variety it is very important to find all open
sets. Since under the above action the variety of Leibniz algebras
consists of orbits of algebras, the description of the variety is
reduced to  find the open orbits. By Noetherian consideration
there are a finite number of open orbits. In any variety of algebras there are
algebras with open orbits (so-called rigid algebras). Thus, the
closure of orbits of rigid algebras gives irreducible components
of the variety. Hence, to describe the variety of algebras it is
enough to describe all rigid algebras.

A powerful tool in the study of a variety of algebras is that for
constructive subsets of algebraic varieties, the closures  with
respect to the Euclidean and the Zariski topologies coincide. In
particular, for an algebraically closed field $F$, the limit in
usual Euclidean topology  leads to the same limit as in the
Zariski topology. It has lead to consideration of such notions as
deformations and degenerations of algebras. In fact, a rigid
algebra is characterized by absence of any degeneration coming from any algebra, that is,
the orbit of the rigid algebra does not belong to the closure of the orbit of any other algebra. Existence or absence of
degeneration in a given variety of algebras is revealed by
construction or by using invariant arguments. This approach is
very effective in case of nilpotent and solvable algebras.

The description of a variety of any class of algebras is a very
difficult problem. Note that for the description of the variety of
nilpotent Lie algebras with dimensions less than eight the works
\cite{Burde1, GrunHall, Seeley} are devoted. The complete
description of orbits closure of four-dimensional Lie algebras is
given in \cite{Burde2}. To the investigation of the variety of
Leibniz algebras  the work \cite{AOR} is devoted. In particular, in
\cite{AOR} it is described all irreducible components of the
varieties of complex nilpotent Leibniz algebras of dimensions less
than 5.

On the other hand, Grunewald and O'Halloran in \cite{GrunHall} proposed the following:

\textbf{Conjecture:} Any $n$-dimensional nilpotent Lie algebra is
a degeneration of some algebra of the same dimension.

In other words there is not nilpotent rigid algebra in the variety
of Lie algebras, although a rigid Lie algebra exists in the
subvariety of nilpotent Lie algebras. The statement is based on the fact
that second cohomology groups of rigid algebras are trivial,
while for nilpotent Lie algebras, they are always nontrivial.
Similarly to the case of Lie algebras, Balavoine proved the general
principles for deformations and rigidity of Leibniz algebras
\cite{Bal}.

In this paper we prove that solvable Leibniz algebras, whose
nilradical is rigid in the variety of nilpotent Leibniz algebras,
cannot be obtained as a degeneration of a solvable Leibniz algebra with different nilradical. In other words,
 any solvable Leibniz algebra with a given rigid nilradical, such that there is not other solvable
  Leibniz algebra with the same nilradical, is rigid.
   The description of solvable Leibniz algebras with
three-dimensional rigid nilradical is obtained. Moreover, we
prove that the Conjecture above is true for dimensions less than six and for
the case of Leibniz algebras the Conjecture is true for dimension
less than four. Finally, we find one algebra which was omitted in
the work \cite{Gorb}.

Throughout the paper we consider finite-dimensional vector spaces
and algebras over the field $\mathbb{C}$. Moreover, in the
 multiplication table of an algebra omitted products are assumed to be
zero and if it is not noticed we shall consider non-nilpotent
solvable algebras.

\section{Preliminaries}

In this section we give necessary definitions and results for
understanding main parts of the work.

\begin{defn}[\cite{Lo}]
A vector space $L$ over a field $F$ with a binary operation
$[-,-]$ is called a Leibniz algebra, if for any $x, y, z \in L$ the
so-called  Leibniz identity  holds
  \[
\big[x, [y, z]\big] = \big[[x, y], z\big] - \big[[x, z], y\big].
\]

\end{defn}

Every Lie algebra is a Leibniz algebra, but the bracket in a
Leibniz algebra needs not be skew-symmetric.

For a Leibniz algebra $L$ consider the following lower central  and
derived series:
\begin{align*}
L^1 & =L, \qquad  L^{k+1}=[L^k,L^1], \qquad \ \ k \geq 1, \\
L^{[1]} & = 1,  \qquad  L^{[s+1]} = [L^{[s]}, L^{[s]}], \qquad s \geq 1.
\end{align*}
\begin{defn}  A Leibniz algebra $L$ is said to be
nilpotent (respectively, solvable), if there exists  $p\in\mathbb
N$ $(q\in \mathbb N)$ such that $L^p=0$ (respectively,
$L^{[q]}=0$).
\end{defn}

It is well known \cite{Ayup} that in Leibniz algebras case, in each dimension, there
exists a unique (up to isomorphism)  algebra with maximal index of
nilpotency  whose  multiplication table is:
\[\mathbf{NF_n}: \ [e_i, e_1] = e_{i+1}, \quad 1 \leq i
\leq n-1.\]

Denote by $\Leib_n$ (respectively, by $\LN_n$  and $\LR_n$) the set
of all $n$-dimensional (respectively, nilpotent and solvable)
Leibniz algebras.

\begin{rem}
Null-filiform Leibniz algebras of dimension $n$ can be characterized as $n$-dimensional nilpotent Leibniz algebras such that the $n$-th term in the lower central series is nontrivial. This means that their orbits are open sets in the variety of $n$-dimensional nilpotent Leibniz algebras with respect to the Zariski topology, hence  null-filiform Leibniz algebras of dimension $n$ are rigid.
\end{rem}

Let $\lambda$ and $\mu$ be Leibniz algebras  of
the same dimension over a field $F$.

\begin{defn}  It is said that an algebra $\lambda$
degenerates to an algebra $\mu$, if $\Orb(\mu)$ lies in the Zariski closure
of $\Orb(\lambda)$, $\overline{\Orb(\lambda)}$. We denote this by $\lambda \rightarrow \mu$.

The degeneration $\lambda \rightarrow \mu $ is called a direct degeneration if there is not a chain of nontrivial
degenerations of the form: $\lambda \rightarrow \nu \rightarrow
\mu$.

The level of an algebra
$\lambda$, denoted by $\lev_n(\lambda)$, is the maximum length of a
chain of direct degenerations, which, of course, ends with the
algebra $\mathbf{a_n}$ (the algebra with zero multiplication).
\end{defn}

\begin{rem}
Recall that any $n$-dimensional algebra degenerates to the algebra
$\mathbf{a_n}$.
\end{rem}

Further we shall use the fact from the algebraic groups theory on
constructive subsets of algebraic varieties that their closures
relative to the Euclidean and the Zariski topologies coincide.
It is not difficult to see that the $\GL_n(\mathbb{C})$-orbits are
constructive sets. Therefore, the usual Euclidean topology on
$\mathbb{C}^{n^3}$ leads to the same degenerations as does the
Zariski topology, that is, the following condition $\lambda\to\mu$
implies that
\[ \text{there exists} \ \  g_t\in \GL_n\big(\mathbb{C}(t)\big) \ \text{ such that } \ \lim_{t\to
0} \ g_t*\lambda=\mu,\] where $\mathbb{C}(t)$ is the field of
fractions of the polynomial ring $\mathbb{C}[t]$.

\begin{rem} It is easy to note that a rigid nilpotent (solvable) algebra
cannot be obtained by degeneration of any other nilpotent
(solvable) algebra.
\end{rem}

Further we shall need the following results.

%\begin{lem} \cite{GrunHall}\label{lemgrum}. If $H$ is an algebraic group over $\mathbb{C}$ which acts
%algebraically on an affine variety $V$ over $\mathbb{C}$, then for
%every $v$ in $V$ $$cl(H \cdot v) = cl^*( H \cdot v),$$ where $cl$
%and $cl^*$ are closures in Zariski and Euclidean topologies,
%respectively.
%\end{lem}

\begin{prop}[\cite{GrunHall}] \label{P:borel} Let $G$ be a reductive algebraic
group over $\mathbb{C}$ with Borel subgroup $B$ and let $X$ be an
algebraic set on which $G$ acts rationally. Then
\[\overline{G * x} = G * \overline{(B * x)} \quad \text{for all} \  x\in X.\]
\end{prop}

Note that for the classification of solvable Leibniz algebras with
given nilradical, it is important the number of nil-independent
derivations of the  nilradical. Namely, for a solvable Leibniz algebra
with nilradical $N$, the dimension of the complementary vector
space to $N$ is not greater than the maximal number of
nil-independent derivations of $N$.

\begin{thm}[\cite{Cas1}] \label{P:NFn} Let $R$ be a solvable Leibniz algebra whose
nilradical is $\mathbf {NF_n}$. Then there exists a basis $\{e_1, e_2,
\dots, e_n, x\}$ of the algebra $R$ such that the multiplication
table of $R$ with respect to this basis has the following form:
\[\mathbf{RNF_n}: \left\{\begin{array}{ll} [e_i,e_1]=e_{i+1}, & 1\leq i\leq n-1,\\[1mm]
[x,e_1]=e_1, & \\[1mm]
[e_i,x]=-ie_i, & 1\leq i\leq n.\end{array}\right.\] \end{thm}

In \cite{GrunHall} it was shown that the rigid nilpotent Lie
algebras in dimensions less than six are the following:
\[\begin{array}{ll}
\mathbf{n_3}:& [e_1, e_2] = -[e_2, e_1]= e_3;\\
\mathbf{n_4}:& [e_1, e_2] = -[e_2, e_1]= e_3, \quad [e_1, e_3] = -[e_3, e_1]= e_4;\\
\mathbf{n_5}:& [e_1, e_2] = -[e_2, e_1]= e_3, \quad [e_1, e_3] = -[e_3, e_1]= e_4, \\
&  [e_1, e_4] = -[e_4, e_1]= e_5, \quad [e_2, e_3] = -[e_3,
e_2]= e_5.
\end{array}\]

Due to \cite{AOR} we can present the list of three-dimensional nilpotent rigid Leibniz algebras:
\[\begin{array}{rl}
 \boldsymbol{\lambda_4}(\alpha): &  [e_1, e_1] =
e_3,\quad   [e_2, e_2] =\alpha e_3,\quad  [e_1, e_2] = e_3, \ \
\alpha\neq 0;\\
 \boldsymbol{ \lambda_5}:& [e_2, e_1]=e_3, \quad [e_1, e_2] = e_3;\\
\boldsymbol{\lambda_6}: &  [e_1, e_1] = e_2, \quad [e_2, e_1] = e_3.
\end{array}\]

%For a given Leibniz algebra $\lambda$ we put:
%\begin{itemize}
%\item $R(\lambda)=\{ x\in\lambda|[\lambda,x]=0\}$ -- the right annihilator of
%$\lambda;$
%\item $L(\lambda)=\{ x\in\lambda|[x,\lambda]=0\}$ -- the left annihilator of
%$\lambda;$
%\item $Z(\lambda)=\{ x\in\lambda|[x,\lambda]=[\lambda,x]=0\}$ -- the center of
%$\lambda;$
%\item $Aut(\lambda)$ -- the group of automorphisms of $\lambda;$
%\item $\lambda^k=[\lambda^{k-1},\lambda]$ -- the $k$-th degree of
%$\lambda;$
%\item $SA(\lambda)$ -- the maximal abelian subalgebra of
%$\lambda;$
%\item $Com(\lambda)$ -- the maximal commutative subalgebra of $\lambda;$
%\item $Lie(\lambda)$ -- the maximal Lie subalgebra of $\lambda.$
%\end{itemize}

%\begin{prop} \cite{AOR}\label{propinvariants}. An algebra $\lambda$ does not degenerate to $\mu$ if one of
%the following conditions is valid:
%\begin{enumerate}
%\item  $dim\lambda^m<dim\mu^m$ for some $\mu,$ \item $
%dimR(\lambda)>dimR(\mu),$ \item $dimL(\lambda)>dimL(\mu),$ \item $
%dimZ(\lambda)>dimZ(\mu),$ \item $dim Aut(\lambda) \ge
%dimAut(\mu),$ \item $dimSA(\lambda)>dimSA(\mu),$ \item
%$dimCom(\lambda)>dimCom(\mu),$ \item
%$dimLie(\lambda)>dimLie(\mu).$
%\end{enumerate}
%\end{prop}

\begin{prop}[\cite{AOR}] \label{P:n2} Let $\lambda$ be a complex non Lie algebra of
$\mathcal{L}eib_n$. Then $\lambda \to  \mathbf{n_2}\oplus \mathbb{C}^{n-2}$, where
$\mathbf{n_2}: [e_1, e_1] =e_2$ is a two-dimensional non-abelian nilpotent
Leibniz algebra.
\end{prop}

Consider the following algebras:
\begin{align*}
\mathbf{p_n ^ {\pm}}: \ [e_1, e_i] &= e_i,   & [e_i, e_1] &= \pm e_i, \qquad i \geq 2, &&  &&  && \\
 \mathbf{n_3^{\pm}}: \ [e_1, e_2]&= e_3,   & [e_2, e_1] &= \pm e_3. && &&  &&
\end{align*}
\begin{thm}[\cite{Gorb}]\label{T:level} Let $\lambda$ be an $n$-dimensional algebra. Then

1. If the algebra $\lambda$ is skew-commutative, then $\lev_n
(\lambda) = 1 $ if and only if it is isomorphic to $\mathbf{p_n^{-}}$ or
 to the algebra $\mathbf{n_3^{-}} \oplus \mathbf{a_{n-3}}$ ($n  \geq 3$). In
particular, the algebra $\lambda$ is a Lie algebra.

2. If the algebra $\lambda$ is commutative, then
$\lev_n(\lambda)=1$ if and only if it is isomorphic to $\mathbf{p_n^{+}}$ or
 to the algebra $\mathbf{ n_3^{+}} \oplus \mathbf{a_{n-3}}$ ($n \geq 3$). In
particular, the algebra $\lambda$ is a Jordan algebra.
\end{thm}
\begin{rem}
We note that the algebra $\mathbf{p_n^{+}}$ is not a Jordan algebra.
\end{rem}

\section{Main results}

We divide the main  section into three subsections where we study the rigidness of solvable Leibniz algebras with rigid nilradical, describe such four-dimensional algebras with three-dimensional radical and present one algebra of level one, which was omitted in the work \cite{Gorb}.

\subsection{Rigidity of solvable Leibniz algebras with rigid nilradical}\

In this subsection we investigate the rigidity of solvable Leibniz
algebras with rigid nilradical.

\begin{defn}
The algebras whose orbits are open sets in the variety $\mathcal{L}eib_n$ with respect to the Zariski topology are said to be rigid.
\end{defn}

\begin{rem}
The notion of rigidity is characterized by absence of any degeneration coming from any algebra, that is, the orbit of the rigid algebra does not belong to the closure of the orbit of any other algebra.
\end{rem}

Let $N$ be a nilpotent Leibniz algebra. Denote by $\LR_n(N)$ the
set of all $n$-dimensional solvable Leibniz algebras whose
nilradical is $N$.

For any $m \ (1 \leq m \leq n)$ define the subset $\wedge_m \subset
\LR_n$ such that $\wedge_m = \{\lambda = (c_{i,j}^k)\}$ with
 the properties:
\begin{align*}
\sum\limits_{k_1=n-m+1}^n \sum\limits_{k_2=n-m+1}^n \dots
\sum\limits_{k_{s-1}=n-m+1}^n c_{i_1, i_2}^{k_1}c_{k_1,
i_3}^{k_3}\dots c_{k_{s-1}, i_s}^{k_s} &=0, \quad n-m+1 \leq i_1, i_2, \dots, i_s \leq n,\\
c_{i,j}^k &=0, \quad 1 \leq i, j \leq n, \ 1 \leq k \leq
n-m,
\end{align*}
where $c_{i,j}^k$ are structural constants and $s$ any fixed number.

Let us observe that $R \in \wedge_m$ if and only if $R$ contains the nilpotent
ideal $N = \langle \{e_{n-m+1}, e_{n-m+2}, \dots, e_{n}\}\rangle$ satisfying
$R^2\subseteq N$.

It is not difficult to see that $\wedge_m$ is a Zariski closed
subset of $\LR_n$, but it is not $\GL_n(\mathbb{C})$-stable.
However, the set $\wedge_m$ is $B$-stable, where $B$ is the  Borel
subgroup of $\GL_n(\mathbb{C})$ consisting of upper triangular
matrices.

\begin{prop} \label{P:dim} Let $R_1, R_2 \in \mathcal{L}R_n$ and let $R_1 \in \mathcal{L}R_n(N_1), \
R_2 \in \mathcal{L}R_n(N_2)$. If $R_1 \rightarrow R_2$, then $\dim N_2 \geq
\dim N_1$. \end{prop}

\begin{proof} Let $\dim N_1=m$, then choose $g \in
\GL_n(\mathbb{C})$ such that $R' = g * R_1 \in \wedge_m$. Since $B
* R' \in \wedge_m$ and $\wedge_m$ is a closed set, then $\overline{B
* R'} \in \wedge_m$. By Proposition \ref{P:borel} and by condition
$R_1 \rightarrow R_2$ we conclude that $R_2 \in \GL_n(\mathbb{C}) *
\wedge_m$. Therefore, the algebra $R_2$ contains a nilpotent ideal
of dimension $m$. Since $N_2$ is the nilradical of $R_2$, we get
$\dim N_2 \geq m$.
\end{proof}

\begin{cor}\label{C:dim} Let $R_1 \in  \mathcal{L}R_n(N_1)$ and $R_2 \in  \mathcal{L}R_n(N_2)$. If $\dim N_1 =
\dim N_2$ and $R_1 \rightarrow R_2$, then $N_1 \rightarrow N_2$.
\end{cor}

\begin{proof} Let  $g_t$ be a family such that
$\displaystyle \lim_{t \to 0} \ g_t \ast R_1 = R_2$. By Proposition \ref{P:dim}
we have that $\displaystyle \lim_{t\to 0} \ g_t \ast N_1$ is a nilpotent
ideal of $R_2$. Therefore, we get $\dim N_1 = \dim \ (\displaystyle \lim_{t \to
0} \ g_t \ast N_1) = \dim N_2$. Since $N_2$ is the nilradical of $R_2$,
then $\displaystyle \lim_{t\to 0} \ g_t \ast N_1 = N_2$, i.e., $N_1
\rightarrow N_2$.
\end{proof}

Consider now a solvable Leibniz algebra $R$ with rigid nilradical
$N$.

\begin{prop}\label{P:R2}  Let $R^2=N$ and suppose that there exists a solvable Leibniz algebra $R_1$ such
that $R_1 \rightarrow R$. Then $R_1 \in \mathcal{L}R_n(N)$.
\end{prop}

\begin{proof} Let $N_1$ be the nilradical of the algebra $R_1$.
Note that by the Proposition \ref{P:dim} $\dim N_1 \leq \dim N$.

If $\dim N_1 < \dim N$, then we have $\dim R_1^2 \leq \dim N_1 < \dim N =
\dim R^2$, which is a contradiction to the condition $R_1
\rightarrow R$ by a consequence of  \cite[Theorem 1.4]{GrunHall} (see also  \cite[Corollary]{AOR}).

If $\dim N_1 = \dim N$, then by  Corollary \ref{C:dim} we conclude that
$N_1 \rightarrow N$. Since $N$ is a rigid algebra, then  we get $N_1
\cong N$.
\end{proof}

\begin{cor}\label{C:RNFn}
 The algebra $\mathbf{RNF_n}$ is a rigid algebra of the variety $\mathcal{L}R_{n+1}$.
\end{cor}

From the above, we conclude that  for a rigid nilpotent Leibniz algebra $N$ in the variety $\LN_s$ and for
 $R \in \LR_n(N)$ there are only two possibilities: $R$ is  rigid in $\LR_n$ or there
exists a rigid algebra $R_1 \in \LR_n(N)$ such that $R_1 \rightarrow
R$.

Next proposition  establishes a relationship between a solvable algebra
and its nilradical.

\begin{prop} For any solvable algebra $R$ with nilradical $N$
there exists a degeneration: $R \rightarrow N \oplus \mathbb{C}^k$,
where $k= \dim R/N$.
\end{prop}
\begin{proof} We choose a basis $\{e_1, \dots, e_k, e_{k+1},
\dots,  e_n\}$ of $R$ such that $N = \langle  \{e_{k+1}, \dots, e_n\} \rangle$. A
degeneration is given by the family $g_t$ defined as follows:

\[
g_t(e_i)=
\begin{cases}
 t^{-1}e_i & \text{if} \quad  1\leq i \leq k, \\
e_i &\text{if} \quad k+1\leq i \leq n.
\end{cases}
\]
Indeed,
\begin{align*}
g_t \ast [e_i, e_j] &= g_t([g_t^{-1}(e_i), g_t^{-1}(e_j)]) = t^2g_t([e_i,
e_j]) =t^2[e_i, e_j] \rightarrow 0, \ \  1 \leq i, j \leq k,\\
g_t \ast [e_i, e_j]&=g_t([g_t^{-1}(e_i), g_t^{-1}(e_j)]) = tg_t([e_i, e_j]) =t[e_i, e_j]
\rightarrow 0,  \ \ 1 \leq i \leq k, \ k+1 \leq j \leq n, \\
g_t \ast [e_i, e_j]&=g_t([g_t^{-1}(e_i),g_t^{-1}(e_j)])=g_t([e_i, e_j])
=[e_i, e_j], \ \ k+1 \leq i, j \leq n. \qedhere
\end{align*}
\end{proof}

Now we present a family, which will be useful in the sequel,
\[g_t(e_1)=t^{-1}e_1, \quad g_t(e_2)=t^{-1}e_2, \quad
g_t(e_i)=t^{-i+1}e_i, \quad 3 \leq i \leq n,\] that degenerates the
algebra $\mathbf{NF_n}$ to the so-called filiform algebra \cite{Ayup}
\[\mathbf{F_n}: [e_i, e_1] = e_{i+1}, \ 2 \leq i \leq n-1.\]

\subsection{Classification of four-dimensional solvable Leibniz algebras with three-dimensional rigid nilradicals}\

In this subsection we classify four-dimensional solvable Leibniz
algebras whose nilradical is rigid and three-dimensional.

First at all, in the following proposition we describe the derivations of the three-dimensional nilpotent rigid Leibniz
algebras $\boldsymbol{ \lambda_4}(\alpha)$, $\boldsymbol{ \lambda_5}$ and $\boldsymbol{ \lambda_6}$. Recall that a derivation of a Leibniz algebra $(L,[-,-])$ is a $F$-linear map $d \colon L \to L$ such that $d[x,y]=[d(x),y]+[x,d(y)]$, for all $x,y \in L$.

\begin{prop}\label{P:der} In the algebras $\boldsymbol{ \lambda_4}(\alpha)$,
$\boldsymbol{ \lambda_5}$ and $\boldsymbol{ \lambda_6}$ there exist bases such that their derivations
have the following forms:
\[\begin{array}{rclrcl}
\Der\big( \boldsymbol{ \lambda_4}(\alpha)\big) & =&
\begin{pmatrix}
a_1 & 0 & a_3 \\
0 & b_2 & b_3 \\
0 & 0 & a_1+b_2
\end{pmatrix},
\ \alpha \neq \frac 1 4; &
\Der\big(\boldsymbol{ \lambda_4}(\frac{1}{4})\big) &=&
\begin{pmatrix}
a_1 & a_2 & a_3 \\
0 & a_1 & b_3 \\
0 & 0 & 2a_1
\end{pmatrix},\\
\Der( \boldsymbol{ \lambda_5})& = &
 \begin{pmatrix}
a_1 & 0 & a_3 \\
 0 & b_2 & b_3 \\
0 & 0 & a_1+b_2
\end{pmatrix}, &
\Der( \boldsymbol{ \lambda_6})
&=& \begin{pmatrix}
 a_1 & a_2 & a_3 \\
0 & 2a_1 & a_2 \\
 0 & 0 & 3a_1
\end{pmatrix}.
\end{array}\]
\end{prop}

\begin{proof} Taking the following change of basis  in the algebra $\boldsymbol{ \lambda_4}(\alpha)$: \[f_1 = e_1, \qquad f_2= e_2 + \beta e_1, \qquad
f_3 = e_3,\] with $\beta = -\frac{1+\sqrt{1-4\alpha}}{2}$, we
deduce that the multiplications of $\boldsymbol{ \lambda_4}(\alpha)$ becomes of the form:
\[ [f_1, f_1] = f_3, \quad [f_2, f_1] =\beta f_3,\quad [f_1, f_2] = (1 + \beta)f_3.\]

If $\beta\neq - \frac{1}{2}$ (i.e., $\alpha\neq \frac{1}{4}$), then
setting $f'_1 = f_1 - \frac{1}{2\beta+1}f_2, \ f_2' =
\frac{1}{\beta}f_2$, we derive
\begin{equation}\label{E:beta}
[f_2, f_1] = f_3, \quad [f_1, f_2] = \beta' f_3,
\end{equation}
where $\beta' = \frac{\sqrt{1-4\alpha}-1}{\sqrt{1-4\alpha}+1}$.

If $\beta = -\frac{1}{2}$ (i.e., $\alpha=\frac{1}{4}$), then
putting $f_2' = -2f_2$, we get
\begin{equation}\label{E:lambda}
 \boldsymbol{ \lambda_4}\Big(\frac{1}{4}\Big): [f_1, f_1]
= f_3, \quad [f_2, f_1] = f_3, \quad [f_1, f_2] = - f_3.
\end{equation}

By checking the derivation property for algebras \eqref{E:beta} and \eqref{E:lambda}
we obtain
\[\Der\big(\boldsymbol{ \lambda_4}(\alpha)\big) = \begin{pmatrix}
a_1 & 0 & a_3 \\
0 & b_2 & b_3 \\
0 & 0 & a_1+b_2
\end{pmatrix},  \ \alpha \neq \frac 1 4; \quad \Der\big(\boldsymbol{ \lambda_4}\big(\frac {1} {4}\big)\big) = \begin{pmatrix}
a_1 & a_2 & a_3 \\
0 & a_1 & b_3 \\
0 & 0 & 2a_1
\end{pmatrix}.\]

The derivations of the algebras $\boldsymbol{ \lambda_5}$ and $\boldsymbol{ \lambda_6}$ are obtained directly
applying the derivation property.
\end{proof}

Below, we prove that there do not exist four-dimensional solvable Leibniz algebras with
nilradical $\boldsymbol{ \lambda_4}(\frac{1} {4})$.

\begin{prop} There are not  four-dimensional solvable Leibniz
algebras with three-dimensional nilradical  $\boldsymbol{ \lambda_4}(\frac {1}
{4})$.
\end{prop}

\begin{proof} Let us assume the contrary. Let $R \in \LR_4\big(\boldsymbol{ \lambda_4}
(\frac 1 4)\big)$. We choose a basis $\{x, f_1, f_2, f_3\}$ of $R$
such that $\{f_1, f_2, f_3\} $ is the basis of $\boldsymbol{ \lambda_4}(\frac {1}
{4})$ chosen in the proof of Proposition \ref{P:der}. Since the
algebra $R$ is non-nilpotent, the restriction of the
right multiplication operator  $\mathcal{R}_{x}$ to $\boldsymbol{ \lambda_4}(\frac{1}{4})$ is a non-nilpotent derivation of $\boldsymbol{ \lambda_4}(\frac{1}{4})$. Then using the
form of this derivation from Proposition \ref{P:der} we have
\begin{align*}
[f_1, x] &= a_1f_1 + a_2f_2 + a_3f_3, & [f_2, x]& = a_1f_2 +
b_3f_3,& [f_3, x]& = 2a_1f_3,\\
[f_1, f_1] &= f_3, & [f_2, f_1] &= f_3, & [f_1, f_2]& = -
f_3.
\end{align*}
Since $\mathcal{R}_{x \mid \boldsymbol{\lambda_4}}$ is non-nilpotent, we can suppose $a_1 =
1$. It is easy to see that the right annihilator of the algebra
$R$ only consists  of $\{f_3\}$. Therefore,
\begin{align*}
[f_1, x] &= f_1 + a_2f_2 + a_3f_3, & [f_2, x] &= f_2 +
b_3f_3,& [f_3, x] &= 2f_3,\\
[x, f_1] &= - f_1 - a_2f_2 + \alpha_3f_3, & [x, f_2] &= - f_2 +
\beta_3f_3,& [x, x] &= \gamma_3f_3,\\
[f_1, f_1] &= f_3, & [f_2, f_1] &= f_3, & [f_1, f_2] &= - f_3.
\end{align*}
Considering the Leibniz identity
\begin{align*}
0 &= \big[x, [f_2, f_1]\big] = \big[[x,
f_2], f_1\big] - \big[[x, f_1], f_2\big]  \\
 {}& =[- f_2 + \beta_3f_3, f_1] -
[- f_1 - a_2f_2 + \alpha_3f_3, f_2] = - f_3 - f_3 = -2 f_3,
\end{align*}
we have a contradiction with the assumption.
\end{proof}

The following theorem gives the classification of four-dimensional
solvable Leibniz algebras with three-dimensional rigid
nilradicals.

\begin{thm} Up to isomorphism, there exist three four-dimensional solvable Leibniz algebras with
three-dimensional rigid nilradicals. Namely,
\[\begin{array}{rl} \mathbf {R_1^4}: & \left\{\begin{array}{rclrclrclrcl}[e_2, e_1] &=& e_3, & [e_1, e_2] &=&\beta
e_3,& [x,e_1]&=&- e_1, & [x, e_2]&=& -\beta e_2,\\[1mm]
[e_1, x] &=& e_1, & [e_2, x] &=&\beta e_2, & [e_3, x] &=& (\beta +
1)e_3,& & &  \end{array} \right.\\
 & where\ \beta =
\frac{\sqrt{1-4\alpha}-1}{\sqrt{1-4\alpha}+1}\ for\ \alpha\neq 0,
\frac{1}{4}; \\

\mathbf {R_2^4}: &  \left\{\begin{array}{rclrclrclrcl}[e_2, e_1] &=& e_3, & [e_1, e_2] &=& e_3,& [x,e_1]&=&- e_1, & [x, e_2]&=& -e_2,\\[1mm]
[e_1, x] &=& e_1, & [e_2, x]& =& e_2, & [e_3, x] &=& 2e_3\,;& & &
\end{array} \right.\\

\mathbf {R_3^4}:&  \left\{\begin{array}{rclrclrcl}[e_1, e_1] &=& e_2, & [e_2, e_1] &=&e_3,& [x,e_1]&=&- e_1, \\[1mm]
[e_1, x] &=& e_1, & [e_2, x] &=& 2e_2, & [e_3, x] &=& 3e_3.
\end{array} \right. \end{array}\]
\end{thm}

\begin{proof}
 Here we shall use the form of the algebra
$\boldsymbol{ \lambda_4}(\alpha)$ as in the proof of Proposition
\ref{P:der} after the change of basis, i.e., the form $\boldsymbol{ \lambda_4}(\beta)$. Consider the class
$\LR_4\big(\boldsymbol{ \lambda_4}(\beta)\big)$. Due to Proposition \ref{P:der}, we can
choose a basis $\{x, f_1, f_2, f_3\}$ of the algebra of
$\LR_4\big(\boldsymbol{ \lambda_4}(\beta)\big)$ such that $\mathcal{R}_{x \mid \boldsymbol{ \lambda_4}(\beta)}$ is a
non-nilpotent derivation of $\boldsymbol{ \lambda_4}(\beta)$. Therefore, in the
algebra of $\LR_4\big(\boldsymbol{ \lambda_4}(\beta)\big)$ we have the following products:
\[\begin{array}{rclrclrcl} [f_2, f_1] &=& f_3, &  [f_1, f_2]& =& \beta f_3, & & &\\[1mm]
[f_1, x] &=& a_1f_1 + a_3f_3, & [f_2, x] &=& b_2f_2 + b_3f_3, & [f_3, x]
&=& (a_1 + b_2)f_3.\end{array}\]

It is easy to see that the right annihilator of the algebra consists
of $\{f_3\}$. Hence we get
\[\begin{array}{rclrclrcl} [f_1, x]& = &a_1f_1 + a_3f_3, & [f_2, x]& =& b_2f_2 + b_3f_3, & [f_3, x]
&=& (a_1 + b_2)f_3,\\[1mm]
[x, f_1] &=& - a_1f_1 + \alpha_3f_3, & [x, f_2] &=& - b_2f_2 +
\beta_3f_3, &  [x, x] &=& \gamma_3f_3,\\[1mm]
[f_2, f_1] &=& f_3, &  [f_1, f_2] &=& \beta f_3. & &  &
\end{array}\]

Applying the Leibniz identity
\begin{align*}
0&= \big[x, [f_2, f_1]\big] = \big[[x, f_2], f_1\big] - \big[[x, f_1], f_2\big]  \\
{}&=[-b_2f_2 + \beta_3f_3, f_1] - [-a_1f_1 + \alpha_3f_3, f_2] = -b_2f_3
+ a_1\beta f_3,
\end{align*}
we derive $b_2 = a_1\beta$.

Since $\mathcal{R}_{x \mid \boldsymbol{\lambda_4}(\beta)}$ is
non-nilpotent, we have $a_1 = b_2\neq 0$. Consequently, we can assume $a_1=1, \ b_2 =\beta$.

Taking the change of basis:
\[e_1 = f_1 - \frac{a_3}{\beta}f_3, \quad e_2=
f_2 - b_3f_3, \quad e_3 = f_3,  \quad x' = x - \frac{\gamma_3}{\beta+1}f_3,\] we can
suppose that $a_3 = b_3 = \gamma_3 =0$ and the
multiplication  table has the form

\[\begin{array}{rclrclrcl} [e_1, x] &=& e_1, & [e_2, x] &=& \beta e_2, & [e_3, x]
&=& (1 + \beta)e_3,\\[1mm]
[x, e_1] &=& - e_1 + \alpha_3e_3, & [x, e_2] &=& - \beta e_2 +
\beta_3e_3, & & &  \\[1mm]
[e_2, e_1] &=& e_3, &  [e_1, e_2] &=& \beta e_3. & & &
\end{array}\]

Consider the chain of equalities
\[\big[x, [e_1, x]\big] = \big[[x, e_1], x\big] - \big[[x, x], e_1\big] = [- e_1 +
\alpha_3e_3, x] = - e_1 + \alpha_3 (1 + \beta)e_3.\]

On the other hand, $\big[x, [e_1, x]\big] = [x, e_1] = - e_1 +
\alpha_3e_3$.

Comparing the coefficients at the basis elements, we obtain
$\alpha_3\beta =0$ which implies $\alpha_3 =0$.

Similarly, from
\begin{align*}
\big[x, [e_2, x]\big] &= \big[[x, e_2], x\big] - \big[[x, x], e_2\big] = [-\beta e_2 +
\beta_3e_3, x] = - \beta^2e_2 + \beta_3(1+ \beta)e_3,\\
\big[x, [e_2, x]\big] &= [x, \beta e_2] = - \beta^2e_2 + \beta \beta_3e_3,
\end{align*}
we deduce $\beta_3 =0$. Thus the algebra $\mathbf {R_1^4}$ is
obtained.

Applying the above arguments for the class $\LR_4(\boldsymbol{ \lambda_5})$ we
derive the  multiplication table:
\[\begin{array}{rclrclrcl} [f_1, x] &=& a_1f_1 + a_3f_3, & [f_2, x] &=& b_2f_2 + b_3f_3, & [f_3, x]
&=& (a_1 + b_2)f_3,\\[1mm]
[x, f_1] &=& - a_1f_1 + \alpha_3f_3, & [x, f_2] &=& - b_2f_2 +
\beta_3f_3, &  [x, x] &=& \gamma_3f_3, \\[1mm]
[f_2, f_1] &=& f_3, &  [f_1, f_2] &=& f_3.  & &&
\end{array}\]

From the chain of equalities
\begin{align*}
0 &= \big[x, [f_2, f_1]\big] = \big[[x, f_2], f_1\big] - \big[[x, f_1], f_2\big] \\
{} & = [- b_2f_2 + \beta_3f_3, f_1] - [-a_1f_1 + \alpha_3f_3, f_2] = -b_2f_3
+a_1f_3,
\end{align*}
we have $b_2 = a_1$.

Since the restriction of the right multiplication operator  on the
element $x$ to $\boldsymbol{ \lambda_5}$ is non-nilpotent, we have $a_1 =b_2\neq
0$ and without loss of generality we can suppose $a_1= b_2 =1$.

Taking the change of basis  \[e_1 = f_1 -a_3f_3, \quad e_2= f_2 - b_3f_3, \quad e_3 = f_3,
\quad x' = x - \frac{\gamma_3}{2}f_3,\] we can suppose that $a_3 =
b_3 = \gamma_3 =0$ and the  multiplication table has the form

\[\begin{array}{rclrclrcl} [e_1, x] &=& e_1, & [e_2, x] &=& e_2, & [e_3, x]
&=& 2e_3,\\[1mm]
[x, e_1] &=& - e_1 + \alpha_3e_3, & [x, e_2]& =& - e_2 +
\beta_3e_3, & & & \\[1mm]
[e_2, e_1] &= &e_3, &  [e_1, e_2] &= &e_3. & & &
\end{array}\]

Applying the Leibniz identity to the brackets $\big[x,[x,e_1]\big]$ and $\big[x,[e_1,x]\big]$ with respect to the above multiplication, we derive that
$\alpha_3 =\beta_3 =0$. Thus, we obtain the algebra $\mathbf{R_2^4}$.

Since an algebra of $\LR_4(\boldsymbol{\lambda}_6)$ is nothing else but the algebra
$\mathbf{RNF_3}$, the algebra $\mathbf{R_3^4}$ is directly followed from Theorem
\ref{P:NFn}.
\end{proof}

It should be noted that thanks to Proposition \ref{P:R2} and Corollary \ref{C:RNFn} the algebras
$\mathbf{R_1^4}, \mathbf{R_2^4}$ and $\mathbf{R_3^4}$ are rigid in the variety $\LR_4$.

The following theorem assert that the Conjecture is true for dimensions
less than six.

\begin{thm}  Any complex nilpotent Lie algebra of dimension less than
six is not rigid in $\mathcal{L}R_n$. \end{thm}

\begin{proof} All solvable Lie algebras of dimension less than six have  the following
 multiplication tables \cite{GK}:
\begin{align*}
 \mathbf{r_3}: \ &  [e_1,e_2]=-[e_2,e_1]=e_1+e_3,  \quad [e_3,e_2]=-[e_2,e_3]=e_3,\\
\mathbf{r_4}: \ & [e_1,e_2]=-[e_2,e_1]=e_1+e_3, \quad [e_1,e_3]=-[e_3,e_1]=e_4, \quad  [e_2,e_3]=-[e_2,e_3]=e_3,\\
\mathbf{r_5}: \ & [e_1,e_2]=-[e_2,e_1]=e_3, \quad  [e_1,e_3]=-[e_3,e_1]=e_2, \\ & [e_1,e_4]=-[e_4,e_1]=e_5,
\quad [e_2,e_3]=-[e_3,e_2]=e_5.
\end{align*}

It is easy to check that

$\mathbf{r_3} \rightarrow \mathbf{n_3}$ via the family $g_t$ defined as follows
\[g_t(e_1)=t^{-1}e_1, \quad g_t(e_2)=t^{-1}e_2, \quad
g_t(e_3)=t^{-2}e_3,\]

$\mathbf{r_4} \rightarrow \mathbf{n_4}$ via the family $g_t$ defined as
\[g_t(e_1)=t^{-1}e_1, \quad  g_t(e_2)=t^{-1}e_2, \quad g_t(e_3)=t^{-2}e_3, \quad
g_t(e_4)=t^{-3}e_4,\]

$\mathbf{r_5} \rightarrow \mathbf{n_5}$ via the family $g_t$ defined as
\begin{align*}
g_t(e_1)& =t^{-1}e_1, \qquad      \; \; \; \quad  g_t(e_2)=t^{-3}e_4, \qquad  g_t(e_3)=t^{-4}e_5,\\
 g_t(e_4)& =-e_2+t^{-2}e_4,  \quad g_t(e_5)=-t^{-1}e_3+t^{-3}e_5. \qedhere
\end{align*}
\end{proof}

\begin{rem} \label{R:Rn}
Consider the following $n$-dimensional solvable Leibniz algebra
\[\mathbf{R_n}: \ [e_1, e_1] = e_2, \quad [e_i, e_1] = e_i + e_{i+1}, \quad 2 \leq i \leq n-1.\]

It is known that the algebra $\mathbf{NF_n}$ is rigid in the variety
of nilpotent Leibniz algebras \cite{Ayup}. However, this algebra it is not
rigid in the variety of solvable Leibniz algebras. Indeed, the
family of basis transformations
\[g_t(e_i)=t^{-i}e_i, \qquad 1 \leq i \leq n,\]
degenerates the algebra $\mathbf{R_n}$ to $\mathbf{NF_n}$.
\end{rem}

Now we present a result which asserts that the Conjecture is true for
the case of Leibniz algebras of dimensions less than four.

\begin{thm}  Any nilpotent Leibniz algebra of dimension less than
four is not rigid.
\end{thm}
\begin{proof}  From \cite{AOR} we have a unique two-dimensional rigid nilpotent
Leibniz algebra $\mathbf{n_2}: [e_1, e_1] = e_2$.
It is easy to check that the algebra $\mathbf{r_2}$ with the  multiplication table $[e_2, e_1] =
e_2$ degenerates to $\mathbf{n_2}$ via the family of transformations:
\[g_t: \quad g_t(e_1)=t^{-1}e_1 - t^{-2}e_2, \qquad
g_t(e_2)=t^{-2}e_2.\]

For the three-dimensional case we have the rigid nilpotent algebras $\boldsymbol{ \lambda_4}(\alpha), \boldsymbol{ \lambda_5}$ and
$\boldsymbol{ \lambda_6}$.

Let us consider the solvable Leibniz algebra
\[\mathbf{ r_{3,2}(\alpha)}:\left\{
\begin{array}{ll} [e_1, e_1] = e_3, & [e_1, e_2] = -(2+\beta)\alpha e_1 + e_2 +e_3, \\[1mm]
[e_2, e_1] = (2+\beta)\alpha e_1 - e_2, & [e_2, e_2] = \alpha
e_3,\quad  \alpha \neq 0,\\[1mm]
[e_3, e_1] =\beta e_3,& [e_3, e_2] = (2+\beta)\beta \alpha
e_3,\end{array}\right.\] where $\beta = \frac {1-4\alpha +
\sqrt{1-4\alpha }} {2\alpha}$.

Then $g_t$ defined as \[g_t(e_1)=t^{-1}e_1, \quad
g_t(e_2)=t^{-1}e_2, \quad g_t(e_3)=t^{-2}e_3\] degenerates the algebra $\mathbf{r_{3,2}(\alpha)}$ to the algebra $\boldsymbol{ \lambda_4}(\alpha)$.

Consider the solvable Leibniz algebra \[\mathbf{r_{3,1}}: \ [e_2, e_1] = - e_2+e_3,
\quad [e_3, e_1] = - 2e_3, \quad [e_1, e_2] = e_2+e_3, \quad [e_2,
e_2] = e_3.\]

Then $\mathbf{r_{3,1}} \rightarrow \boldsymbol{ \lambda_5}$ via $g_t$, which is given by \[g_t(e_1)=t^{-1}e_1, \qquad g_t(e_2)=t^{-2}e_2, \qquad
g_t(e_3)=t^{-3}e_3.\]

Due to Remark \ref{R:Rn}, we get $\mathbf{R_3} \rightarrow \boldsymbol{ \lambda_6}$.
\end{proof}

\subsection{On the algebra of level one}\

In this subsection we show that the result of Theorem
\ref{T:level} is not complete. Namely, the algebra $\mathbf{n_2} \oplus \mathbf{a_{n-2}}$
is also an algebra of level one and it is not isomorphic to the algebras
$\mathbf{p_n^{\pm}}$ and $\mathbf{n_3^{\pm}}\oplus \mathbf{a_{n-3}}$.

\begin{thm} The $n$-dimensional commutative algebra $\mathbf{n_2} \oplus
\mathbf{a_{n-2}}$ is of level one.
\end{thm}

\begin{proof} Firstly, we shall prove that the algebra $\mathbf{n_2}\oplus
\mathbf{a_{n-2}}$ does not degenerate to $\mathbf{p_n^{\pm}}$ and $\mathbf{n_3^{\pm}}\oplus
\mathbf{a_{n-3}}$. Since $\mathbf{n_2}\oplus \mathbf{a_{n-2}}$ is commutative, it is enough
to prove it for $\mathbf{p_n^{+}}$ and $\mathbf{n_3^{+}}\oplus \mathbf{a_{n-3}}$.

Let us assume the contrary, that is there exists a family $g_t \in \GL_n(\mathbb{C})$ such
that $\mathbf{n_2}\oplus \mathbf{a_{n-2}} \rightarrow \mathbf{p_n^{+}}$. Let $g_t$ be of the form
\[g_t(e_i) = \sum\limits_{s=1}^n \alpha_{i, s}(t)e_s, \quad g^{-1}_t(e_i) = \sum\limits_{s=1}^n \beta_{i, s}(t)e_s.\]

Consider $g_t(e_2) = \displaystyle \sum_{i=1}^{n}\alpha_{2, i}(t)e_i$. We choose
numbers $p, q$ ($p\neq q$) such that \[\lim_{t\to 0}\frac
{\alpha_{2, p}(t)} {\alpha_{2, q}(t)} < \infty.\]

Consider
\begin{equation}\label{E:gt}
g_t([g^{-1}_t(e_1), g^{-1}_t(e_p)]) =\beta_{1,1}(t)
\beta_{p,1}(t)g_t(e_2) = \beta_{1,1}(t)
\beta_{p,1}(t)\sum_{i=1}^{n}\alpha_{2, i}(t)e_i.
\end{equation}

Since in the algebra $\mathbf{p_n^{+}}$ we have $[e_1, e_p] = e_p$, then $\displaystyle \lim_{t\to 0}g_t([g^{-1}_t(e_1),
g^{-1}_t(e_p)])=e_p$. Therefore, we obtain
\[\lim_{t \to 0} \beta_{1,1}(t) \beta_{p,1}(t)\alpha_{2, p}(t) =1, \quad
\lim_{t \to 0} \beta_{1,1}(t) \beta_{p,1}(t)\alpha_{2,
q}(t) =0.\]

On the other hand,
\begin{align*}
\lim_{t \to 0} \beta_{1,1}(t) \beta_{p,1}(t)\alpha_{2, p}(t) & =
\lim_{t \to 0} \beta_{1,1}(t) \beta_{p,1}(t)\alpha_{2,
q}(t) \frac{\alpha_{2, p}(t)} {\alpha_{2, q}(t)} \\
& = \lim_{t \to 0} \beta_{1,1}(t) \beta_{p,1}(t)\alpha_{2, q}(t) \cdot \lim_{t \to 0}\frac{\alpha_{2, p}(t)} {\alpha_{2, q}(t)} = 0.
\end{align*}

This is a contradiction with the assumption of the existence of $g_t$, i.e.,
$\mathbf{n_2}\oplus \mathbf{a_{n-2}}$ does not degenerate to $\mathbf{p_n^+}$.

Let us show that $\mathbf{n_2} \oplus \mathbf{a_{n-2}}$ does not degenerate to the algebra $\mathbf{n_3^{+}}\oplus \mathbf{a_{n-3}}$. Similarly as above we can assume the existence of a family $g_t$.

From \eqref{E:gt} we get \[g_t([g^{-1}_t(e_1),
g^{-1}_t(e_1)])=\beta_{1,1}(t)
\beta_{1,1}(t)\sum_{i=1}^{n}\alpha_{2, i}(t)e_i.\]

Since in the algebra $\mathbf{n_3^{+}}\oplus \mathbf{a_{n-3}}$ we have the product
$[e_1, e_1] = 0$, then \[\lim_{t\to 0}g_t([g^{-1}_t(e_1),
g^{-1}_t(e_1)])=0.\] Consequently,  $\displaystyle \lim_{t \to
0}\beta_{1,1}(t) \beta_{1,1}(t)\alpha_{2, 3}(t)=0$.

Similarly, from \eqref{E:gt} with $p=2$, i.e.,  \[g_t([g^{-1}_t(e_1),
g^{-1}_t(e_2)]) = \beta_{1,1}(t)
\beta_{2,1}(t)\sum_{i=1}^{n}\alpha_{2, i}(t)e_i\] and of the product
$[e_1, e_2] = e_3$ in $\mathbf{n_3^{+}}\oplus \mathbf{a_{n-3}}$, we conclude
 \[\lim_{t \to 0}\beta_{1,1}(t)\beta_{2,1}(t)\alpha_{2, 3}(t)=1.\]

Using the equalities
\begin{align*}
g_t([g^{-1}_t(e_2), g^{-1}_t(e_2)]) & =
g_t\big(\big[\sum\limits_{s=1}^n \beta_{2, s}(t)e_s, \sum\limits_{s=1}^n
\beta_{2, s}(t)e_s\big]\big) \\
&=\beta_{2,1}(t) \beta_{2,1}(t)g_t(e_2) = \beta_{2,1}(t)
\beta_{2,1}(t)\sum_{i=1}^{n}\alpha_{2, i}(t)e_i
\end{align*}
and $[e_2, e_2] = 0$, we derive $\displaystyle \lim_{t \to 0}\beta_{2,1}(t)
\beta_{2,1}(t)\alpha_{2, 3}(t)=0$.

Thus, we summarize
\[\lim_{t \to 0}
\beta_{1,1}(t) \beta_{1,1}(t)\alpha_{2, 3}(t)=\lim_{t \to
0} \beta_{2,1}(t) \beta_{2,1}(t)\alpha_{2, 3}(t)=0, \quad \lim_{t
\to 0} \beta_{1,1}(t) \beta_{2,1}(t)\alpha_{2, 3}(t)=1.\]

However,
\[\lim_{t \to 0}\big(\beta_{1,1}(t) \beta_{2,1}(t)\alpha_{2, 3}(t)\big)^2 =\lim_{t \to 0}\beta_{1,1}(t)
\beta_{1,1}(t)\alpha_{2, 3}(t) \cdot \lim_{t \to
0}\beta_{2,1}(t) \beta_{2,1}(t)\alpha_{2, 3}(t)=0.\]

Thus, the algebra $\mathbf{n_2} \oplus \mathbf{a_{n-2}}$ does not degenerate to $\mathbf{n_3^{+}}\oplus \mathbf{a_{n-3}}$.

Now we shall prove that  $\lev_n(\mathbf{n_2}\oplus \mathbf{a_{n-2}})=1$. Assume that there exists a Leibniz algebra $\lambda$ such that
$\mathbf{n_2}\oplus \mathbf{a_{n-2}} \rightarrow \lambda$ is a direct degeneration.
Then $\dim \Orb(\lambda) < \dim \Orb(\mathbf{n_2}\oplus \mathbf{a_{n-2}})$ (see
\cite{Gorb}).

If $\lambda$ is a non-Lie Leibniz algebra, then by Proposition
\ref{P:n2} we have that  $\lambda \rightarrow \mathbf{n_2}\oplus
\mathbf{a_{n-2}}$. Then there exists a chain of direct degenerations
$\lambda \rightarrow \lambda_1 \rightarrow \dots \rightarrow
\lambda_k \rightarrow \mathbf{n_2}\oplus \mathbf{a_{n-2}}$. Again by \cite{Gorb}, we
have that $\dim \Orb( \mathbf{n_2}\oplus \mathbf{a_{n-2}})<\dim \Orb(\lambda_k) < \dots < \dim
\Orb(\lambda_1) < \dim \Orb(\lambda)$. This is a contradiction with $\dim \Orb(\lambda)
< \dim \Orb(\mathbf{n_2}\oplus \mathbf{a_{n-2}})$.

Let $\lambda$ be a Lie algebra, then by assumption there exists a family
$g_t$ such that \[\lim_{t\to 0} \ g_t \ast (\mathbf{n_2}\oplus \mathbf{a_{n-2}})=
\lambda.\]
 Then from the following equalities
\[g_t([g^{-1}_t(e_i), g^{-1}_t(e_j)]) =
g_t\big(\big[\sum\limits_{s=1}^n \beta_{i, s}(t)e_s, \sum\limits_{s=1}^n
\beta_{j, s}(t)e_s\big]\big) = \beta_{i,1}(t) \beta_{j,1}(t)g_t(e_2),\] we
deduce $\lambda(e_i, e_j) = \lambda(e_j, e_i)$. Since $\lambda$ is a Lie
algebra, it follows that it is abelian. Consequently, the algebra
$\mathbf{n_2}\oplus \mathbf{a_{n-2}}$ is of level one.
\end{proof}

\section*{ Acknowledgements}

 The first and third authors were supported by Ministerio
de Ciencia e Innovaci\'on (European FEDER support included), grant
MTM2009-14464-C02, and by Xunta de Galicia, grant Incite09 207 215
PR. The fourth author was partially supported by the Grant (RGA) No:11-018 RG/Math/AS\_I--UNESCO FR: 3240262715.

\end{document}